\documentclass[12pt]{article}
\usepackage[utf8]{inputenc}
\usepackage{amsfonts}
\usepackage{amsmath}
\usepackage{amsthm}
\usepackage{amssymb}
\usepackage{mathrsfs}
\usepackage{mathtools}
\usepackage{tikz-cd}
\newcommand{\mmax}{\mathfrak{M}_2}
\newcommand{\mints}{\mathfrak{O}_2}
\newcommand{\kints}{\mathfrak{O}_0}
\newcommand{\kmax}{\mathfrak{M}_0}
\newcommand{\asscorder}{\mathfrak{A}_{K_2/K_0}}
\newcommand{\bigfield}{K_2}
\newcommand{\midfield}{K_1}
\newcommand{\littlefield}{K_0}
\newcommand{\ek}{e_{0}}
\newcommand{\pk}{\pi_{0}}
\newcommand{\pim}{\pi_2}
\newcommand{\el}{e_1}
\newcommand{\rz}{\rho_{0}}
\newtheorem{theorem}{Theorem}[section]
\newtheorem{proposition}[theorem]{Proposition}
\newtheorem{lemma}[theorem]{Lemma}
\newtheorem{corollary}[theorem]{Corollary}

\newtheorem{choice}[theorem]{Choice}

\title{Galois Scaffolds and Galois Module Structure for Totally Ramified $C_{p^2}$-extensions
in Characteristic $0$}
\author{Kevin Keating \\
Department of Mathematics \\
University of Florida \\
Gainesville, FL 32611 \\
USA \\[.2cm]
{\tt keating@ufl.edu}
\and 
Paul Schwartz \\
Department of Mathematics \\
University of Florida \\
Gainesville, FL 32611 \\
USA \\[.2cm] 
{\tt paulschwartz@ufl.edu}}
\begin{document}

\maketitle 

\begin{abstract}
Recently, much work has been done to investigate
Galois module structure of local field extensions, particularly through the use of
Galois scaffolds. Given a totally ramified
$p$-extension of local fields $L/K$, a Galois
Scaffold gives us a $K$-basis for $K[G]$ whose effect on the
valuation of elements of $L$ is easy to determine.\\ 
\indent In 2013, N.P. Byott and G.G. Elder
gave sufficient conditions for the existence of Galois
scaffolds for cyclic extensions of degree $p^2$ in
characteristic $p$. We take their work and adapt it to 
cyclic extensions of degree $p^2$ in characteristic $0$.
\end{abstract}
\section{Introduction}
When studying the Galois module structure for a Galois extension
of local fields $L/K$, a useful tool one has is,
as described in [BCE],
a $K$-basis 
for $K[G]$ ($G=Gal(L/K)$) whose effect on the valuation of
elements of $L$ is easy to determine. 
This in essence is a Galois scaffold. In [BE13] Byott and Elder 
gave sufficient conditions for the existence a of Galois scaffold
for totally ramified extensions of degree $p^2$ of local fields
of characteristic $p$. Given an extension $L/K$, which satisfies their assumptions,  
the lower ramification numbers $b_1,b_2$ fall into one 
residue class modulo $p^2$ represented by $0\leq b<p^2$. They conclude that
$\mathfrak{O}_L $ is free over its associated order $\mathfrak{A}_{L/K}$ if and only if $b\mid p^2-1$.
Furthermore, if $\mathfrak{O}_L$ is free over $\mathfrak{A}_{L/K}$
then any element $\rho\in L$ with $v_L(\rho)=b$ satisfies 
$\mathfrak{O}_L=\mathfrak{A}_{L/K}\cdot\rho $.
We translate their work into the setting of characteristic $0$. Thanks are due to 
Griff Elder for referring us to the work of Vostokov and Zhukov on Artin-Schrierer-Witt extensions in characteristic $0$.
\subsection{Local Fields and Ramification}
Given a local field $K$ we let $v_K:K \to \mathbb{Z}\cup\{\infty\}$
be the normalized valuation on $K$ (this will always mean $v_K(0)=\infty)$. The \textit{ring of integers} of $K$
is $\mathfrak{O}_K=\{x\in K:v_K(x)\geq 0\}$ and the
unique maximal ideal of $\mathfrak{O}_K$ is 
$\mathfrak{M}_K=\{x\in K:v_K(x)\geq 1 \} $. We denote by 
$\pi_K$ a \textit{uniformizer} for $K$. This is an element that satisfies
$v_K(\pi_K)=1$. Hence $\mathfrak{M}_K=(\pi_K)$. We let
$e_K=v_K(p)$ where $p$ is the characteristic of the \textit{residue field}
$\mathfrak{O}_K/\mathfrak{M}_K$ of $K$. We define the \textit{Artin-Schrierer
map} $\wp:K\to K$ by $\wp(x)=x^p-x$ and we let $K^{alg}$ denote an algebraic closure of $K$. 

Let $K_n/\littlefield$ be a totally ramified Galois
extension of degree $p^n$. Let $G=Gal(K_n/\littlefield)$ 
and let $v_n:K_n\to\mathbb{Z}\cup\{\infty\}$ be the normalized valuation
on $K_n$. 
For each integer $i$ let $G_i=\{\sigma\in G:v_n((\sigma-1)\pi_n\geq i+1 \}$. 
Observe that $G_{i+1}\subseteq G_i $ and $G_{0}=G$.
We say that
$G_i$ is the $i$th group in the (lower)
ramification filtration of $G$. 
It is known that $G_{i}$ is a normal subgroup of 
$G$ for each $i$ and each quotient $G_i/G_{i+1}$ is elementary abelian.
So we may choose a composition
series $\{1\}=H_n\subseteq H_{n-1}\subseteq...
\subseteq H_1\subseteq H_0=G $ that refines the ramification filtration with $H_i/H_{i+1}\cong C_p$. For each
$1\leq i\leq n$ choose $\sigma_i\in H_{i-1}\setminus H_i $, let
$K_i=K_n^{H_i} $ be the fixed field of $H_i$.
We let $v_i$, $\mathfrak{O}_i$, and
$\mathfrak{M}_i$ denote the
normalized valuation, ring of integers,
and unique maximal ideal for $K_i$ respectively.
Additionally $e_i=e_{K_i}$ and $\pi_i=\pi_{K_i}$ for
$0\leq i\leq n $.

Let $b_i=
v_n((\sigma_i-1)\pi_n)-1$. This gives us a non-decreasing list of integers 
$b_1\leq b_2\leq...\leq b_n$ which are independent of the choices made. Note
that $b_i$
is the called the \textit{$i$th} (lower) \textit{ramification number}. Note that
$b_{j+1},...,b_n$ are the lower ramification
numbers for $K_n/K_j$ and $b_1,...,b_j$ are the
lower ramification numbers for $K_j/K$ [BE18, pg. 101].

We define the \textit{upper ramification numbers} 
$u_1=b_1 $ and $$u_{i}=u_{i-1}+\frac{b_i-b_{i-1}}{p^{i-1}} $$ for $2\leq i\leq n$.
Note that $u_1,...,u_j $ are the upper ramification
numbers for $K_j/K$ but $u_{j+1},...,u_n $
are not necessarily the upper ramification numbers
for $K_n/K_j$ [BE18, pg. 101].
Suppose $E/K$ is a Galois $p$-extension and $F/K$ is a Galois
subextension.
It is well known that if $u$ is an upper ramification number for
$F/K$ then $u$ is also an upper ramification number for
$E/K$ as well.

\subsection{Depth of Ramification}
We have the following definition from Hyodo [H].
Let $L$ be a finite extension of $K$.
For finite $M/L$, and $F\in \{M,L,K\}$ 
define the depth of ramification (with respect to $F)$ by
$$d_F(M/L):=
\inf\{v_F\left(Tr_{M/L}(y)\right/y):y\in M\setminus\{0\} \}. $$
It is elementary to see that $d_F(M/L)\geq 0 $.
Hyodo points out that 
\begin{equation}\label{different}
d_F(M/L)=v_F(\mathfrak{D}_{M/L})-v_F(\pi_L)+v_F(\pi_M) 
\end{equation}
where $\mathfrak{D}_{M/L}$
is the different for $M/L$.
So if $M/K$ is a totally ramified $C_{p^2}$ extension we see that
$$d_M(M/K)=(p-1)b_2+p(p-1)b_1. $$
It follows from (\ref{different}) that
\begin{center}
$d_F(M/L)=d_F(M/N)+d_F(N/L)$
\end{center}
for any intermediate field $N$.
\subsection{Galois Scaffolds for $C_{p^2}$-extensions}
Let $\littlefield$ be a local field whose residue field has characteristic
$p$. Let $\bigfield/\littlefield$ be a totally ramified $C_{p^2}$-extension such that the lower ramification numbers
$b_1$ and $b_2$ are relatively prime to $p$ and fall into
one residue class modulo $p^2$ represented by $0< b<p^2$.
Let $G=Gal(\bigfield/\littlefield)\cong C_{p^2}$ and let $\midfield/\littlefield$ be the unique $C_p$-subextension.
Set $\mathbb{S}_{p^2}=\{0,1,...,p^2-1 \}$ and define a function $\mathfrak{a}:\mathbb{Z}\to \mathbb{S}_{p^2} $
by $\mathfrak{a}(j)\equiv jb_2^{-1} \mod p^2.$ For $0\leq i\leq 1$
let $\mathfrak{a}(j)_{(i)} $ denote the $i$-th digit in the 
$p$-adic expansion of $\mathfrak{a}(j) $.

Given an integer $\mathfrak{c}\geq 1$,  
two things are required for a \textit{Galois
scaffold of precision $\mathfrak{c}$}
[BCE, Definition 2.3]:\\
1. For each $t\in\mathbb{Z}$ an element $\lambda_t\in \bigfield $
such that $v_2(\lambda_t)=t $ and $\lambda_s\lambda_t^{-1}\in \littlefield$
whenever $s\equiv t\mod p^2.$\\
2. Elements $\Psi_1,\Psi_2$ in the augmentation ideal
$(\sigma-1:\sigma\in G) $ of $\littlefield[G]$  such that for each
$1\leq i\leq 2$ and $t\in\mathbb{Z}$
$$\Psi_i\lambda_t\equiv \left\{\begin{array}{ccc}
u_{i,t}\lambda_{t+p^{2-i}b_i} & \mod 
\lambda_{t+p^{2-i}b_i}\mmax^\mathfrak{c} & 
\mbox{if } \mathfrak{a}(t)_{(2-i)}\geq 1\\
0 & \mod \lambda_{t+p^{2-i}b_i}\mmax^\mathfrak{c}
& \mbox{if } \mathfrak{a}(t)_{(2-i)}=0
\end{array}\right. $$
where $u_{i,t}\in K$ and $v_K(u_{i,t})=0$.
\section{Witt Vectors in Characteristic 0}
\subsection{Witt Vectors of length 2}
A thorough treatment of the Witt Ring is given in
chapter 1 of [FV]. Here we shall state the relevant
information.

Let $B$ be a commutative ring with unity. 
Let $$S_1(X_1,Y_1)=X_1+Y_1 $$
$$S_2(X_1,X_2,Y_1,Y_2)=X_2+Y_2+
\frac{X_1^p+Y_1^p-(X_1+Y_1)^p}{p}. $$
Let the \textit{Witt vectors of length 2 over $B$} be the set
$W_2(B)=B\times B$ with addition defined by
$$(a_1,a_2)\oplus (b_1,b_2)=
(S_1(a_1,b_1),S_2(a_1,a_2,b_1,b_2)). $$
Define the \textit{Frobenius map} $\mathbf{F}:W(B)\to W(B) $
by $\mathbf{F}(a_1,a_2)=(a_1^p,a_2^p). $ 
The map $\wp=\mathbf{F}-id $ (Witt vector subtraction) is called the
\textit{Artin-Schrier} operator.
Let $$D(X,Y):=\frac{X^p+Y^p-(X+Y)^p}{p}=
-\sum_{i=1}^{p-1}\frac{\binom{p}{i}}{p}X^{p-i}Y^i\in\mathbb{Z}[X,Y]. $$
Observe that $S_2(X_1,X_2,Y_1,Y_2)=X_2+Y_2+D(X_1,Y_1).$

\subsection{Cyclic extensions of degree $p^2$}
From here onward, $\littlefield$ is assumed to be a local field of characteristic $0$
with residue characteristic $p$. Fix an algebriac closure $\littlefield^{alg}$
of $\littlefield$.
We will use the following to help build our scaffold.
\begin{theorem}[VZ, Proposition 3.2]\label{russians} 
Let $a_1\in \littlefield$, $-\frac{p}{p^2-1}\ek<v_0(a_1)\leq 0$. 
Also let $a_2\in \littlefield$ with $v_0(a_1)+v_0(a_2)>-\frac{p}{p-1}\ek $.
Put $\bigfield=\littlefield(x_1,x_2)$
for $\wp(x_1,x_2)=(a_1,a_2)$ where $\wp:
W_2(\littlefield^{alg})\to W_2(\littlefield^{alg})$ is the Artin-Schrier operator. 
Then, if $x_1\notin \littlefield$, $\bigfield/\littlefield$ is a cyclic 
extension of degree $p^2$ and
\begin{center} 
$d_{\littlefield}(\bigfield/\littlefield)<\frac{p^2+1}{p^2+p}\ek. $
\end{center}
\end{theorem}

\section{Building The Scaffold}
Here we use Witt vectors to construct totally ramified $C_{p^2}$-extensions 
which possesses a Galois scaffold.\\ 
\begin{choice}\label{first choice} Choose $a_1\in \littlefield\setminus \wp(\littlefield)$ such that
$p\nmid v_0(a_1)$ and $-\frac{p}{p^2-1}\ek<v_0(a_1)<0$.
\end{choice}
\begin{choice}\label{second choice} Choose $\mu\in \littlefield$ such that $m:=-v_0(\mu)>0$ 
satisfies: 
\begin{alignat}{2}
\frac{p}{p-1}\ek & >pm-\left(2+\frac{1}{p(p-1)}\right)v_0(a_1) \label{first bound} \\
   p^2m & >-(p^2-1)v_0(a_1). \label{second bound}
\end{alignat}
\end{choice}
Set $a_2=\mu^pa_1$. Notice that (\ref{first bound}) tells us that
\begin{equation}\label{first bound equals}
\frac{p\ek}{p-1}>-v_0(a_2)-\left(1+\frac{1}{p(p-1)}\right)v_0(a_1).
\end{equation}

Choose $x_1,x_2\in \littlefield^{alg}$ such that
$\wp(x_1,x_2)=(a_1,a_2) $. That is to say
$x_1^p-x_1=a_1 $ and $x_2^p-x_2=a_2+D(x_1,a_1). $
Let 
$$\begin{array}{ccc}
    \midfield=\littlefield(x_1) & \quad\text{and}\quad & \bigfield=\littlefield(x_1,x_2). 
\end{array}$$
Observe that $K_1\neq \littlefield$ since
$a_1\in \littlefield\setminus \wp(\littlefield)$.
Let $K_2=\littlefield(x_1,x_2)$, it follows from Theorem \ref{russians}
that $K_2/\littlefield$ is a $C_{p^2}$-extension.
Since $-\frac{p}{p-1}\ek<v_0(a_1)<0$ and
$p\nmid v_0(a_1)$ it follows from [MW, Theorem 5] that
$\midfield/\littlefield$ is a totally ramified $C_p$-extension with
ramification number $u_1=b_1=-v_0(a_1)$.
The goal of this section is to show that 
$\bigfield/\littlefield$ has a Galois scaffold.

Set $$D_1:=D(x_1,a_1)=\frac{x_1^p+a_1^p-(x_1+a_1)^p}{p}=
-\sum_{i=1}^{p-1}\frac{\binom{p}{i}}{p}x_1^ia^{p-i} 
\in \midfield.$$
Since $$\min\{-ib_1-p(p-i)b_1:1\leq i\leq p-1 \}=-(p^2-p+1)b_1$$ 
we deduce that $$v_1(D_1)=-(p^2-p+1)b_1. $$

\begin{lemma}\label{upper breaks}
Let $E/L$ be a totally ramified $C_p\times C_p$-extension with
upper ramification numbers $u_1<u_2$. There is a unique $C_p$-subextension
$F/L$ of $E/L$ with ramification number $u_1$. All other $C_p$-subextensions of $E/L$ have ramification
number $u_2.$
\begin{proof}
Let $G=Gal(E/L)$. Consider the upper ramification filtration
$$G^x=\left\{\begin{array}{ccc}
G, & \mbox{if} & x\leq u_1\\
G^{u_2}, & \mbox{if} & u_1<x\leq u_2\\
\{1\}, & \mbox{if} & u_2<x.
\end{array}\right.$$
Let $H\leq G$ such that $|H|=p$. We see that
$$\left(G/H\right)^x=G^xH/H=
\left\{\begin{array}{ccc} 
G/H,& \mbox{if} & G^x\not\subseteq H\\
\{\overline{1}\}, & \mbox{if} & G^x\subseteq H.
\end{array}\right. $$
If $H=G^{u_2}$ then $G^x\leq H$ if and only if $G^x\leq G^{u_2}$,
which is equivalent to $x>u_1$. In this case $u_1$ is the upper
ramification number for $E^H/L$.
If $H\neq G^{u_2}$ then $G^x\leq H$ if and only if
$G^x=\{1\}$, which occurs exactly when $x>u_2$. Thus $u_2 $
is the upper ramification number for $E^H/L$.
Set $F:=E^{G^{u_2}}$. Then $F/L$
is the unique $C_p$-subextension of $E/L$
with ramification number $u_1$. 
\end{proof}
\end{lemma}

\begin{proposition}\label{breaks}
$K_2/\littlefield$ is a totally ramified extension
with lower ramification numbers $b_1=-v_0(a_1)$ and 
$b_2=p^2m+b_1$. The upper ramification numbers for $K_2/\littlefield$ are
$u_1=b_1 $ and $u_2=-v_0(a_2)=pm+b_1$.
\end{proposition}
\begin{proof}
Let $z_1,z_2\in \littlefield^{alg}$ satisfy $\wp(z_1)=a_2 $ and
$\wp(z_2)=D_1$. Let $N=\littlefield(z_1)$ and $M=\midfield(z_2)$. 
It follows from (\ref{first bound equals}) that $v_0(a_2)>-\frac{p\ek}{p-1}$. Since
$p\nmid v_0(a_2)$ it follows from [MW, Theorem 5] that $N/\littlefield$ is a totally ramified 
$C_p$-extension with ramification number $-v_0(a_2)$.
Likewise $p\nmid v_1(D_1)$ and
$-v_0(a_1)<\frac{p\ek}{p^2-1}<\frac{p^2\ek}{(p-1)(p^2-p+1)}$. Thus
$v_1(D_1)>-\frac{p\el}{p-1}$. Hence 
$M/\midfield$ is a totally ramified $C_p$-extension with ramification number
$-v_1(D_1)$. Additionally, $M=\littlefield(x_1,z_2)$
with $\wp(x_1,z_2)=(a_1,0) $ so it follows from
Theorem \ref{russians} that $M/\littlefield$ is a $C_{p^2}$-extension.
Now the lower ramification numbers for $M/\littlefield$ are
$-v_0(a_1)=b_1$ and $-v_1(D_1)=(p^2-p+1)b_1 $.
So the upper ramification numbers for $M/\littlefield$ are
$-v_0(a_1)$ and $-v_0(a_1)+p^{-1}(v_0(a_1)-v_1(D_1))=-p v_0(a_1) $.

Now $E:=MN$
is a totally ramified $C_{p^2}\times C_p$-extension of $\littlefield$ with upper ramification
numbers $-v_0(a_1),-pv_0(a_1)$ and $-v_0(a_2)$.
It follows from (\ref{second bound}) that $-p v_0(a_1)<-v_0(a_2)$
so the upper ramification numbers for $E/\littlefield$
are ordered by $u_1'=-v_0(a_1)$, $u_2'=-pv_K(a_1)$
and $u_3'=-v_0(a_2)$. Thus the lower ramification numbers for
$E/\littlefield $ are $b_1'=-v_0(a_1)$, $b_2'=-v_1(D_1)$, and
$b_3'=p^2m-v_0(a_1)$.
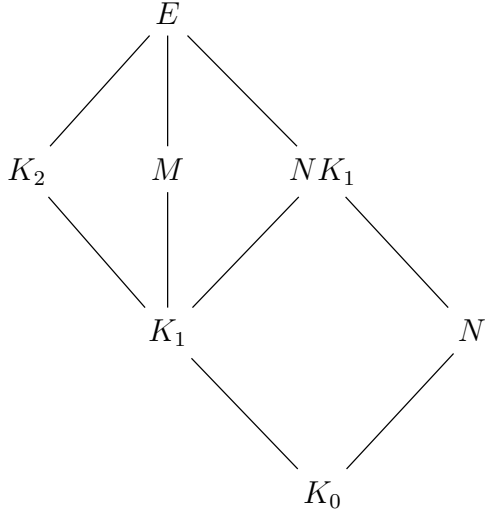
\begin{center}
\begin{tikzcd}[row sep=huge]
& E \ar[dl,-]\ar[d,-]\ar[dr,-] & & \\
K_2 \ar[dr,-] & M \ar[d,-] & N\midfield  \ar[dl,-] \ar[dr,-] & \\
& \midfield \ar[dr,-] & & N \ar[dl,-]\\
& & \littlefield &
\end{tikzcd}
\end{center}

Set $f(X)=X^p-X-a_2-D_1\in E[X]$. Let $$g(X)=f(X+z_1+z_2)=
X^p-X+\sum_{i=1}^{p-1}\binom{p}{i}(z_1+z_2)^{p-i}X^i+
\sum_{i=1}^{p-1}\binom{p}{i}z_1^{p-i}z_2^i. $$
Observe that $v_E(z_1)=p^2 v_0(a_2)$ and $v_E(z_2)=p v_1(D_1)$. It follows
from (\ref{first bound}) 
that $(p^2-p+1)b_1<p^2m+b_1 $ which implies that $v_E(z_1+z_2)=
p^2v_0(a_2)$. Now (\ref{first bound equals}) implies
\begin{align}
v_E\left(\binom{p}{i}z_1^{p-1}z_2\right) &
=
p^3e_K+(p-1)p^2v_K(a_2)+p(p^2-p+1)v_0(a_1) \nonumber \\
& =p^2(p-1)\left(\frac{p}{p-1}\ek+v_0(a_2)
+\left(1+\frac{1}{p(p-1)}\right)v_0(a_1)\right) \nonumber \\
& >0. \nonumber 
\end{align}
Now $g(X)\in \mathfrak{O}_E[X]$, $v_E(g(0))>0$ and $v_E(g'(0))=0$. Thus
we may apply Hensel's lemma and choose $\alpha\in \mathfrak{O}_E$ such that
$g(\alpha)=0$. This means that $\wp(z_1+z_2+\alpha)=a_2+D_1 $
and so $K_2\subseteq E$. Hence $K_2/\littlefield$ is totally ramified.

Now $E/\midfield$ is a totally ramified $C_p\times C_p$ extension 
with lower ramification numbers $-v_1(D_1)=(p^2-p+1)b_1 $ and $p^2m+b_1$.
Since $K_2\neq M$ and $-v_1(D_1) $ is the ramification number for $M/\midfield$
it follows from the previous lemma
that $p^2m+b_1 $ is the ramification number for 
$K_2/\midfield$ and thus the second (lower)
ramification number for $K_2/\littlefield$.
The second upper ramification number is 
$u_2=u_1+p^{-1}(u_2-u_1)=pm+b_1=-v_0(a_2)$.
\end{proof}
Observe that proposition \ref{lower breaks} implies that
assumption \ref{one residue class} is satisfied.
Additionally,
it follows from (\ref{second bound}) and Proposition \ref{breaks} that 
\begin{equation}\label{lower breaks}
    p^2b_1<b_2.
\end{equation}
We also note that in the proof of Proposition 3.2 we showed that
\begin{equation}\label{valuation of second generator}
   v_2(x_2)=\min\{-(p^2-p+1)b_1,-pb_1-p^2m\}=-pb_1-p^2m=-(p-1)b_1-b_2.
\end{equation}
Lastly, it follows from Theorem \ref{russians} and
(\ref{different}) that
\begin{equation}\label{depth of ramification}
d_{K_2}(K_2/\littlefield)=
(p-1)b_2+p(p-1)b_1<\left(\frac{p^2+1}{p^2+p}\right)(p^2\ek).
\end{equation}
The following is presumably well known:
\begin{lemma}\label{keatings lemma}
Let $F$ be a local field of characteristic
$0$. Let $\beta\in F\setminus \wp(F)$ such that $u=-v_F(\beta)$
satisfies $0<u<\frac{p}{p-1}e_F $. Let $\alpha$ be a root of $f(X)=X^p-X-\beta $ and set
$E=F(\alpha)$. If $f(X)$ is irreducible then $E/F$ is
a $C_p$-extension and there
is $\sigma\in Gal(E/F)$ such that $\sigma(\alpha)=\alpha+1+
\epsilon$, with $v_F(\epsilon)\geq e_F-(1-\frac{1}{p})u$.
Moreover, if $p$ does not divide $u$ then
$f(X)$ is irreducible and
$v_F(\epsilon)=e_F-(1-\frac{1}{p})u$.
\end{lemma}
\begin{proof}
Since $f(\alpha)=0$ we have
\begin{align}
f(X+\alpha) & =(X+\alpha)^p-(X+\alpha)-\beta \nonumber \\
& =X^p-X+\sum_{i=1}^{p-1}\binom{p}{i}\alpha^{p-i}X^i \nonumber 
\end{align}
with $$v_F\left(\binom{p}{i}\alpha^{p-i}\right)=
e_F-(p-i)p^{-1}u>0$$
for $1\leq i\leq p-1$. It follows from Hensel's Lemma that
$f(X)$ has $p$ roots $\alpha_0,\alpha_1,...,\alpha_{p-1}$
in $E$, with $\alpha_i\equiv \alpha+i\mod 
\mathfrak{M}_E^{pe_F-(p-1)u}$.
Hence $Gal(E/F)\cong C_p $. We have $\alpha_0=\alpha $
and $\alpha_1=\alpha+1+\epsilon $ for some 
$\epsilon\in E$ with $v_F(\epsilon)
\geq e_F-(1-\frac{1}{p})u $.

Now assume $p$ does not divide $u=-v_F(\beta)$.
Let $e(E/F) $ and $f(E/F)$ 
be the \textit{ramification index} and \textit{residue degree} of
$E/F$ respectively.
The equations $p v_E(\alpha)=v_E(\beta)=e(E/F)v_F(\beta) $ 
and $p=e(E/F)f(E/F) $
imply that $e(E/F)=p $. Hence
$E/F$ is totally ramified. 

Let $\sigma\in Gal(E/F)$ be such that $\sigma(\alpha)=
\alpha+1+\epsilon$. Then for $i\geq 1$ we have 
$$\sigma^i(\alpha)=\alpha+i+\epsilon+\sigma(\epsilon)+...+
\sigma^{i-1}(\epsilon).$$
In particular, we get $$\alpha=\sigma^p(\alpha)=\alpha+p+
Tr_{E/F}(\epsilon)$$ and $Tr_{E/F}(\epsilon)=-p$.
It follows from the definition of depth of ramification that
$$v_E(p^{-1}\epsilon)\leq -d_E(E/F)=
-(p-1)u. $$
Thus $v_F(\epsilon)\leq e_F-(1-\frac{1}{p})u$ and so
$v_F(\epsilon)=e_F-(1-\frac{1}{p})u $.
\end{proof}

\begin{proposition}\label{Witt vector stuff}
Let $C_1=D(x_1,1)=\frac{x_1^p+1-(x_1+1)^p}{p}\in \midfield $. There is
$\sigma_1\in Gal(K_2/\littlefield)$ such that
\begin{align}
(\sigma_1-1)x_1 & =1+\epsilon \nonumber \\
(\sigma_1-1)x_2 & =C_1+\delta' \nonumber
\end{align}
where $v_2(\epsilon)=p^2\ek-p(p-1)b_1$ and $v_2(\delta')>0 $.
\end{proposition}
\begin{proof}
Let $\vec{X}=(X_1,X_2) $
and $\vec{a}=(a_1,a_2)$. 
Let $\vec{x}=(x_1,x_2)\in W_2(K_2)$. Notice that 
$\vec{x}$ is a solution to the equation
$\mathbf{F}(\vec{X})=\vec{X}\oplus\vec{a}. $
Replacing $\vec{X}$ with $\vec{X}\oplus\vec{x} $ we get a new equation 
$\mathbf{F}(\vec{X}\oplus\vec{x})=\vec{X}\oplus\vec{x}\oplus 
\vec{a}=\vec{X}\oplus\mathbf{F}(\vec{x}). $
So
$$
((X_1+x_1)^p,(X_2+x_2+D(X_1,x_1))^p)=
(X_1+x_1^p,X_2+x_2^p+D(X_1,x_1^p)).
$$
Lemma \ref{keatings lemma} tells us that
the first coordinate equation is solved by $X_1=1+\epsilon$
with $v_L(\epsilon)=p\ek-(p-1)b_1$. Substitute 
$X_1=1+\epsilon$ into the second coordinate
equation
\begin{equation}\label{sub}
(X_2+x_2+D(1+\epsilon,x_1))^p=X_2+x_2^p+D(1+\epsilon,
x_1^p).
\end{equation}

For $\alpha\in K_2$ with $v_2(\alpha)<0$ we have,
$$D(1+\epsilon,\alpha)=-\sum_{i=1}^{p-1}\frac{\binom{p}{i}}{p}\alpha^{p-i}
(1+\epsilon)^i=
D(1,\alpha)-\sum_{i=1}^{p-1}\left(\frac{\binom{p}{i}}{p}\alpha^{p-i}\sum_{j=1}^i
\binom{i}{j}\epsilon^j\right).$$
For $1\leq i\leq p-1$ we have
$$v_2\left(\sum_{j=1}^{i}\binom{i}{j}\epsilon^j\right)=v_2(\epsilon)>0. $$
Since $v_2(\alpha)<0$ it follows that
$$v_2\left(\sum_{i=1}^{p-1}\left(\frac{\binom{p}{i}}{p}\alpha^{p-i}\sum_{j=1}^i
\binom{i}{j}\epsilon^j\right)\right)=v_2(\epsilon\alpha^{p-1}). $$
So 
\begin{align}
D(1+\epsilon,x_1) & \equiv D(1,x_1) \mod \mmax^{p^2\ek-2p(p-1)b_1}\nonumber \\
D(1+\epsilon,x_1^p) & \equiv D(1,x_1^p) \mod
\mmax^{p^2\ek-p(p^2-1)b_1}.
\nonumber \end{align} 
Recall $b_1=-v_0(a_1)<\frac{p}{p^2-1}\ek$ so
$p^2\ek-2p(p-1)b_1>p^2\ek-p(p^2-1)b_1>0 $.
Now from (\ref{sub}) we get the congruence
\begin{equation}\label{simplify}
(X_2+x_2+D(1+\epsilon,x_1))^p\equiv X_2+x_2^p+D(1,x_1^p)
\mod\mmax[X_2].
\end{equation}

Next we show that 
\begin{equation}\label{to be shown}
(X_2+x_2+D(1+\epsilon,x_1))^p\equiv 
(X_2+x_2+D(1,x_1))^p\mod \mmax[X_2]. 
\end{equation}
Set $\beta=x_2+D(1,x_1)$ and
$\beta'=x_2+D(1+\epsilon,x_1)$.
Notice that (\ref{valuation of second generator}) implies that $v_2(x_2)=-(p-1)b_1-b_2$ and (\ref{lower breaks}) implies that
$-(p-1)b_1-b_2<-p(p-1)b_1=v_2(D(1,x_1))=v_2(D(1,x_1+\epsilon)). $
Hence
$v_2(\beta)=v_2(x_2)=v_2(\beta'). $
So $v_2(p\beta^{p-1})=v_2(p(\beta')^{p-1})=
p^2\ek+p(p-1)v_0(a_2)>0 $ by (\ref{first bound equals}).
Thus 
\begin{alignat}{3}
(X_2+\beta)^p & \equiv  X_2^p+\beta^p  && \mod\mmax[X_2] \nonumber \\
(X_2+\beta')^p & \equiv 
X_2^p+(\beta')^p && \mod\mmax[X_2].
\nonumber
\end{alignat}
Now since $v_2\left(\displaystyle\binom{p}{i}x_2^{p-i}\right)>0 $
for $1\leq i\leq p-1$ and $D(1+\epsilon,x_1)\equiv 
D(1,x_1)\mod \mmax$ we see that
$\beta^p\equiv (\beta')^p\mod\mmax $.
Thus (\ref{to be shown}) is proven.
Now (\ref{simplify}) can be restated as
\begin{equation}\label{simp}
    (X_2+x_2+D(1,x_1))^p\equiv X_2+x_2^p+D(1,x_1^p)
    \mod \mmax[X_2].
\end{equation}

It follows from (\ref{depth of ramification}) that $v_2(px_2^{p-1}D(1,x_1))=
p^2\ek-(p-1)^2b_1-(p-1)b_2>p^2\ek-p(p-1)b_1-(p-1)b_2>0$.
Thus 
\begin{alignat}{3}
(X_2+x_2+D(1,x_1))^p & =(X_2+\beta)^p \nonumber \\
& \equiv X_2^p+\beta^p && \mod\mmax \nonumber \\
& \equiv X_2^p+x_2^p+(D(1,x_1))^p+\displaystyle\sum_{i=1}^{p-1}
\binom{p}{i}(D(1,x_1))^{p-i}x_2^i && \mod\mmax \nonumber \\
& \equiv X_2^p+x_2^p+(D(1,x_1))^p && \mod\mmax.
\nonumber
\end{alignat}
So (\ref{simp}) simplifies to 
\begin{equation}\label{simplify again}
    X_2^p+(D(1,x_1))^p\equiv X_2+D(1,x_1^p)\mod\mmax[X_2].
\end{equation}
\indent Suppose $p=2$, then $$(D(1,x_1))^p-D(1,x_1^p)=
2x_1^2\equiv 0\mod\mmax^{4\ek-4b_1}. $$
Now assume $p$ is odd. Consider $$(D(1,x_1))^p\equiv 
-\sum_{i=1}^{p-1}\left(\frac{\binom{p}{i}}{p}\right)^px_1^{pi}
\mod px_1^{(p-1)^2+(p-2)}\mints.$$
Since $$\left(\frac{\binom{p}{i}}{p}\right)^p\equiv 
\frac{\binom{p}{i}}{p}\mod\mmax^{p^2\ek} $$
we get 
$$-\sum_{i=1}^{p-1}\left(\frac{\binom{p}{i}}{p}\right)^px_1^{pi}\equiv 
D(1,x_1^p)\mod \mmax^{p^2\ek-p^2(p-1)b_1}. $$
Notice that $\min\{p^2\ek-p(p-1)^2b_1-p(p-2)b_1,
p^2\ek-p^2(p-1)b_1 \}>p^2\ek-p(p^2-1)b_1>0. $ Thus
$$(D(1,x_1))^p\equiv D(1,x_1^p)\mod\mmax[X_2]. $$
Hence (\ref{simplify again}) simplifies to 
$$X_2^p\equiv X_2\mod\mmax[X_2]. $$

Let $$f(X_2)=(X_2+x_2+D(1+\epsilon,x_1))^p-X_2-x_2^p+D(1+\epsilon,
x_1^p).$$
Notice that $f(X_2)=0$ is equivalent to $X_2$ being a solution
to (\ref{sub}). 
It follows from Hensel's lemma that there is
$\alpha\in \mmax$ with $f(\alpha)=0$. So 
$(x_1+1+\epsilon,x_2+D(1+\epsilon,x_1)+\alpha) $ is a solution to
$\mathbf{F}(\vec{X})=\vec{X}\oplus \vec{a}. $
Since $D(1+\epsilon,x_1)\equiv D(1,x_1)\mod\mmax$ and $D(1,x_1)=C_1,$
we may say $x_2+D(1+\epsilon,x_1)+\alpha=x_2+C_1+\delta' $ 
where $v_2(\delta')>0$. Now $(x_1+1+\epsilon,x_2+C_1+\delta') $
is a solution to $\mathbf{F}(\vec{X})=\vec{X}\oplus \vec{a} .$
Hence there is $\sigma_1\in G$ such that $\sigma_1x_1=x_1+1+\epsilon $
and $\sigma_1x_2=x_2+C_1+\delta'.$
\end{proof}
Now we have $$\sigma_1(x_1,x_2)\equiv
(x_1,x_2)
\oplus 
(1,0)
\mod W_2(\mmax)
$$ 
so $$\sigma_1^p(x_1,x_2)\equiv (x_1,x_2)\oplus
p(1,0)
\equiv (x_1,x_2)\oplus (0,1)
\mod W_2(\mmax).
$$
Thus
$$\sigma_1^px_2\equiv x_2+1\mod\mmax.$$
Let $\sigma_2=\sigma_1^p$. 
Since $\littlefield(x_1)$ is the fixed field of $\langle \sigma_2\rangle\leq G$
we see that
$
    (\sigma_2-1)x_1=0
.$
Also Lemma \ref{keatings lemma} tells us  
$(\sigma_2-1)x_2=1+\delta
$
for some $\delta\in \bigfield$
with $v_2(\delta)\geq p^2\ek+(p-1)v_1(a_2)$.

Given $x,y\in \littlefield^{alg}$ we have 
\textit{truncated exponentiation}
given by $$(1+x)^{[y]}=\sum_{i=0}^{p-1}\binom{y}{i}x^i$$
where 
$\binom{y}{i}=\frac{y(y-1)\cdots(y-i+1)}{i!}. $
\begin{theorem}\label{create the scaffold} Let $\bigfield/\littlefield$ be the $C_{p^2}$-extension
constructed using Choices \ref{first choice} and \ref{second choice}.
There is a Galois scaffold for $\bigfield/\littlefield$
of precision $$\mathfrak{c}\geq
\min\{b_2-p^2b_1,p^2\ek-(p-1)b_2-p(p-1)b_1\}$$
with
$\Psi_1$ and $\Psi_2$ defined by 
$$\Psi_1+1=\sigma_1\sigma_2^{[\mu]}=
\sigma_1\displaystyle\sum_{i=0}^{p-1}\binom{\mu}{i}
(\sigma_2-1)^i $$ and $\Psi_2=\sigma_2-1 $. 
\end{theorem}
\begin{proof}
We follow the construction given in [BE18].

Let $y_1=x_1$ and $y_2=x_2-\mu x_1$. We claim that $v_2(y_2)=-b_2$.
Observe that $ v_2(y_2^p-x_2^p+(\mu x_1)^p)\geq v_2(p)-p(pb_1+p^2m)=
p^2\ek-p^2b_1-p^3m $. This means that 
$\wp(y_1)\equiv \wp(x_2)-\wp(\mu x_1)\mod \mmax^{
p^2\ek-p^2b_1-p^3m} $. Consider, $\wp(\mu x_1)=
\mu^p x_1^p-\mu x_1=\mu^p(x_1+a_1)-\mu x_1=\wp(\mu)x_1+\mu^p a_1=
\wp(\mu)x_1+a_2$, so $\wp(y_1)\equiv
D_1-\wp(\mu)x_1 \mod \mmax^{p^2\ek-p^2b_1-p^3m}$.
It follows from (\ref{lower breaks}) that $v_2(D_1-\wp(\mu)x_1)=-pb_2$.
Since $b_1<\frac{p}{p-1}\ek $, we get $-pb_1-p^3m<p^2\ek-p^2b_1-p^3m $. 
Hence
$v_2(D_1-\wp(\mu)x_1)<p^2\ek-p^2b_1-p^3m $. Thus
$v_2(\wp(y_1))=v_2(D_1-\wp(\mu)x_1)$.
Hence $v_2(y_1)=\frac{1}{p}v_2(D_1-\wp(\mu)x_1)=
-b_2 $. Therefore $y_1$ and $y_2$ satisfy choice 2.3 in [BE18].

As per (5) in [BE18], for $1\leq i\leq j\leq 2$ we choose
$\mu_{i,j}\in \littlefield$ and $\epsilon_{i,j}\in K_{j-1}$ such that
$$(\sigma_i-1)y_j=\mu_{i,j}+\epsilon_{i,j} $$ with
$v_j(\mu_{i,j})=b_i-b_j<v_j(\epsilon_{i,j})$.
Observe $(\sigma_1-1)y_1=1+\epsilon $ so we may set
$\mu_{1,1}=1 $, $\epsilon_{1,1}=\epsilon$.
Next, $$(\sigma_1-1)y_2=(\sigma_1-1)x_2-\mu-\mu\epsilon. $$
Notice that $v_2(\mu)=b_1-b_2 $. In addition, Proposition 
\ref{Witt vector stuff} tells
us that $v_2((\sigma_1-1)x_2)=v_2(C_1)=-p(p-1)b_1 $ and 
$v_2(\mu\epsilon)=p^2\ek-p(p-1)b_1+b_1-b_2 $. It follows from
(\ref{depth of ramification}) that $v_2(\mu\epsilon)>v_2(C_1)$. Also (\ref{lower breaks}) tells us that
$b_2>(p^2-p+1)b_1 $ and so $b_2-b_1>p(p-1)b_1 $. Now we may set
$\mu_{1,2}=-\mu $ and $\epsilon_{1,2}=(\sigma_1-1)x_2-\mu\epsilon$.
Note that $v_2(\epsilon_{1,2})=-p(p-1)b_1$.
Finally, $(\sigma_2-1)y_2=1+\delta. $
Recall that Lemma \ref{keatings lemma} tells us $v_2(\delta)\geq 
p^2\ek-p(p-1)u_2. $
Set $\mu_{2,2}=1$
and $\epsilon_{2,2}=\delta$. 

In order to satisfy assumption 2.9 in [BE18] we must
show that 
$$\mathfrak{c}:=\min_{1\leq i\leq j\leq 2}\{v_2(\epsilon_{i,j})-v_2(\mu_{i,j})-pu_i+p^{2-j}b_i\}>0.$$
We find that \begin{align}
\mathfrak{c} & \geq \min\{p^2\ek-p(p-1)b_1,
b_2-p^2b_1,p^2\ek-(p-1)b_2-p(p-1)b_1\}
\nonumber \\
& =
\min\{b_2-p^2b_1,p^2\ek-(p-1)b_2-p(p-1)b_1\}
\nonumber 
\end{align}
which is positive by (\ref{lower breaks}) and 
(\ref{depth of ramification}).

As per definition 2.7 in [BE18], let $\Psi_1=\sigma_1\sigma_2^{[\mu]}-1$
and $\Psi_2=\sigma_2-1$. It follows from [BE18, Theorem 2.10] 
that there are $\{\lambda_t\}_{t\in\mathbb{Z}} $ which along with
$\Psi_1 $ and $\Psi_2$ give us a Galois scaffold of precision
$\mathfrak{c}.$
\end{proof}
\begin{corollary}\label{scaffold applied to valuation criterion}
For each $0\leq i,j\leq p-1$ 
$$v_2\left(\Psi_1^i\Psi_2^j\alpha\right)=
v_2(\alpha)+jb_2+ipb_1 $$ whenever $\alpha\in \bigfield $
satisfies $v_2(\alpha)\equiv b_2\mod p^2.$
\end{corollary}
\begin{proof}
See Theorem A.1 part (3) in [BCE18].
\end{proof}
Much work has been done to study the Galois module structure using the integer $\mathfrak{c}$. The reader 
should refer to Theorem 3.1 in [BCE] to learn more.
We will take the route that Byott and Elder did in [BE13] and use 
Corollary 3.6 
to derive results about the Galois module structure.
This allows us to reach the conclusions given in
[BCE, Theorem 3.1] while relaxing the assumptions about
precision.
We believe that our theorems could be generalized to allow
$a_2=\mu^pa_1+\epsilon$ for some $\epsilon\in\littlefield$ 
of sufficiently large valuation as in [BE13]. However, we have not
carried out the necessary computations. 
\section{Resulting Galois Module Structure}
Let $\bigfield/\littlefield$ be the $C_{p^2}$-extension constructed in
Section 3 using Choices \ref{first choice} and \ref{second choice},
and let $(\{\lambda_t\}_{t\in\mathbb{Z}},\{\Psi_i \}_{i=1,2}) $ be the Galois
scaffold from Theorem \ref{create the scaffold}.
For every non-negative integer $a$ let 
$$\Psi^{(a)}=\left\{\begin{array}{cc} \Psi_2^{a_{(1)}}\Psi_1^{a_{(0)}}, &
a<p^2\\ 
0, & \mbox{otherwise}
\end{array}\right. $$
where $a=\displaystyle\sum_{i=0}^\infty a_{(i)}p^i$ with $0\leq a_{(i)}\leq p-1$.
Also define a function $\mathfrak{b}$ from the non-negative integers to
$\mathbb{Z}\cup\{\infty\}
$ by $$\mathfrak{b}(a)=\left\{\begin{array}{cc}
(1+a_{(1)})b_2+a_{(0)}pb_1, & a<p^2\\
\infty, & \mbox{otherwise}
\end{array}\right. $$
Note that $\Psi^{(a)}$ and $\mathfrak{b}$ are different from their 
counterparts in
[BCE]. 
It follows from Corollary \ref{scaffold applied to valuation criterion} 
that given any $\rho\in\bigfield$ with 
$v_2(\rho)=b_2$ we have $v_2(\Psi^{(a)}\rho)=\mathfrak{b}(a)$.
For $0\leq a<p^2$, set
$$
d_a=\left\lfloor \frac{\mathfrak{b}(a)}{p^2}\right\rfloor,
$$
so $\mathfrak{b}(a)=d_ap^2+r(\mathfrak{b}(a)) $ where $r(\mathfrak{b}(a))$ is the least nonnegative
residue modulo $p^2$ of $\mathfrak{b}(a)$. 
In addition say $d_a=\infty$ when $a\geq p^2$.
For $0\leq j<p^2$ let 
$$w_j=\min\{d_{j+a}-d_a:0\leq a<p^2-j \}.$$
Observe that $w_j\leq d_j-d_0 $ for all $j$.\\
\indent Let $\rz\in \bigfield$ with $v_2(\rz)=r(b_2)$. Set $\rho=\pk^{d_0}\rz $, so
$v_2(\rho)=b_2$. Moreover, for $a\geq 1 $ set
$$
\rho_a=\pk^{-d_a}\Psi^{(a)}\rho.
$$
Now $\rho_a=0 $ whenever $a\geq p^2$ and $v_2(\rho_a)=r(\mathfrak{b}(a)) $ when $0\leq a<p^2$.
Thus $\{v_2(\rho_a):0\leq a<p^2\}=\{0,1,...,p^2-1\} $, so $\{\rho_a\}_{0\leq a<p^2} $ is a 
$\kints$-basis for $\mints $ and the elements $\Psi^{(a)}\rho $ span $\bigfield$ over $\littlefield$. By
comparing dimensions we see that $\rho$ generates a normal basis for the extension $\bigfield/\littlefield$, and
$\{\Psi^{(a)}\}_{0\leq a<p^2}$ is a $\littlefield$-basis for the group algebra $\littlefield[G]$.\\
\indent We aim to estimate valuations of $\Psi_1^p\rho$
and $\Psi_2^p\rho$. First we will consider $\Psi_2^p\rho$. If $p=2$ and $x\in \bigfield^\times$ 
then
$\Psi_2^p x=(\sigma_2-1)^2 x=
(1-2\sigma_2+1) x=2(1-\sigma_2) x$
and so $v_2(\Psi_2^2 x)\geq 4\ek+b_2+v_2(x) $.
Now assume $p\neq 2$ and let $x\in \bigfield^\times$.
Consider $(\sigma_2-1)^p x=
\left(\sigma_2^p+\sum_{i=1}^{p-1}\binom{p}{i}(-1)^{p-i}
\sigma_2^i-1\right) x=
\sum_{i=1}^{p-1}\binom{p}{i}(-1)^{p-i}\sigma_2^i x$
and so $v_2(\Psi_2^p x)\geq p^2\ek+v_2(x) $. In both
cases we may conclude that
\begin{equation}\label{gets big}
    v_2(\Psi_2^p x)\geq p^2\ek+v_2(x)
\end{equation}
for all $x\in \bigfield^\times$. This of course
implies that
\begin{equation}\label{sigma 2}
v_2(\Psi_2^p\rho)\geq p^2\ek+b_2.
\end{equation}
\indent Now we turn our attention to $\Psi_1^p$.
Recall that $\Psi_1+1=\sigma_1\sigma_2^{[\mu]} $.
Using the binomial theorem we find that
$$(\sigma_1\sigma_2^{[\mu]})^p\rho=
\Psi_1^p\rho+\rho+\sum_{i=1}^{p-1}\binom{p}{i}\Psi_1^i\rho$$
and so
$$(\Psi_1^p+1)\rho\equiv (\sigma_1\sigma_2^{[\mu]})^p\rho
\mod \mmax^{p^2\ek+b_2+pb_1}.$$
We would like to get a lower bound for
$v_2(((\sigma_2^{[\mu]})^p-1)\rho)$.
Using the multinomial theorem we see that
$$(\sigma_2^{[\mu]})^p=
\displaystyle \sum_{\substack{i_0+i_1+...+i_{p-1}=p\\
0\leq i_j\leq p}}\left(\binom{p}{i_0, i_1,..., i_{p-1}}
\prod_{j=0}^{p-1}\left(
\binom{\mu}{j}^{i_j}(\sigma_2-1)^{ji_j}\right)
\right).
$$
Notice that for $0\leq i,j\leq p-1$,
$v_2\left(\binom{\mu}{j}^i(\sigma_2-1)^{ij}\rho \right)\geq ij(b_1-b_2)+(ij+1)b_2=b_2+ijb_1 $
and so it follows that
$$(\sigma_2^{[\mu]})^p\rho\equiv 
\displaystyle\sum_{j=0}^{p-1}\binom{\mu}{j}^p(\sigma_2-1)^{pj}\rho \mod\mmax^{p^2\ek+b_2}.$$
Now (\ref{gets big}) tells us $$((\sigma_2^{[\mu]})^p-1)\rho\equiv
\displaystyle\sum_{j=1}^{p-1}\binom{\mu}{j}^p
(\sigma_2-1)^{pj}\rho\equiv 0\mod
\mmax^{p^2\ek+b_2+p(b_1-b_2)}.$$

What we've shown is that  
\begin{equation}\label{sigma 1}
\Psi_1^p\rho\equiv (\sigma_1^p-1)\rho\equiv  \Psi_2\rho \mod\mmax^{p^2\ek+pb_1-(p-1)b_2}.
\end{equation}
Notice that $p^2\ek+pb_1-(p-1)b_2=p^2\ek+b_2-p^3m\equiv
b_2\mod p^2$. Hence we may apply Corollary \ref{scaffold applied to valuation criterion}, which implies
\begin{equation}\label{scaffold}
\Psi_1^{p+r}\Psi_2^s\rho\equiv 
\Psi_1^r\Psi_2^{s+1}\rho \mod 
\mmax^{p^2\ek+(r+1)pb_1-(p-1-s)b_2} 
\end{equation}
whenever $0\leq r\leq p-1 $ and $0\leq s\leq p-1$.
\begin{proposition}\label{new congruence}
If $c\in \littlefield$ and $v_0(c)\geq d_0-d_j$ then
\begin{equation}\label{important congruence}
c\Psi^{(j)}\rho_r\equiv c\pk^{d_{j+r}-d_r}\rho_{j+r} 
\mod\mmax^{p^2\ek-pb_2-(p^2-p+1)b_1}
\end{equation}
whenever $0\leq j,r<p^2$ and $j+r<p^2.$ Eqaulity
holds in (\ref{important congruence}) if there is not a carry
when adding the $p$-adic expansions of $j $ and $r$.
\end{proposition}
\begin{proof}
First notice that $\Psi^{(j)}\rho_r=
\pk^{-d_r}\Psi^{(j)}\Psi^{(r)}\rho $ and
$\pk^{d_{j+r}-d_r}\Psi_{j+r}=
\pk^{-d_r}\Psi^{(j+r)}\rho. $
Notice that $\Psi^{(j+r)}\rho=\Psi^{(j)}\Psi^{(r)}\rho $ if
there is not a carry when adding the $p$-adic expansions
of $j$ and $r$ in which case $c\Psi^{(j)}\rho_r= c\pk^{d_{j+r}-d_r}\rho_{j+r}. $

Now assume there is a carry when adding the $p$-adic expansions
of $j$ and $r$. Let $j=j_{(0)}+p j_{(1)}$ and $r= r_{(0)}+p r_{(1)} $.
We have $j_{(0)}+r_{(0)}\geq p $ and 
$j_{(1)}+r_{(1)}+1<p. $
Now $$\Psi^{(j+r)}\rho=
\Psi_1^{j_{(0)}+r_{(0)}-p}\Psi_2^{j_{(1)}+r_{(1)}+1}\rho. $$
It follows from (\ref{sigma 1}) and (\ref{scaffold}) that
\begin{align}
\Psi^{(j)}\Psi^{(r)}\rho &=\Psi_1^{j_{(0)}+r_{(0)}-p}\Psi_2^{j_{(1)}+r_{(1)}}\Psi_1^p\rho \nonumber \\
& \equiv \Psi^{(j+r)}\rho\mod 
\mmax^{p^2\ek-(p-1)b_2+pb_1+(j_{(1)}+r_{(1)})b_2+(j_{(0)}+r_{(0)}-p)pb_1} \nonumber 
\end{align}
Observe that $p^2\ek-(p-1)b_2+pb_1+(j_{(1)}+r_{(1)})b_2+(j_{(0)}+r_{(0)}-p)pb_1=p^2\ek-(p+1)b_2-p(p-1)b_1+\mathfrak{b}(j)+\mathfrak{b}(r). $
Now since $-d_r\geq -\frac{\mathfrak{b}(r)}{p^2} $ 
and $v_0(c)\geq d_0-d_j\geq m-\frac{\mathfrak{b}(j)}{p^2} $ we see
that 
$$c\Psi^{(j)}\rho_r\equiv 
c\pk^{d_{j+r}-d_r}\rho_{j+r}\mod 
\mmax^{p^2\ek-pb_2-(p^2-p+1)b_1}. $$
\end{proof}

Observe that if 
\begin{equation}\label{second new bound}
   p^2\ek-(p+1)b_2+(p-1)b_1>0 
\end{equation}
then $p^2\ek-(p-1)b_2+pb_1>2b_2$ and (\ref{sigma 1})
implies that $v_2(\Psi_1^p\rho)=2b_2. $

\begin{proposition}\label{big enough valuation}
Assume (\ref{second new bound}) is satisfied.
It is the case that $$\pk^{d_0-d_j}
\Psi^{(j)}\rho_r\in\mmax $$
whenever $0\leq r,j<p^2$ and $r+j\geq p^2.$
\end{proposition}
\begin{proof}
First observe that 
$$\pk^{d_0-d_j}\Psi^{(j)}\rho_r=
\pk^{d_0-d_j-d_r}\Psi^{(j)}\Psi^{(r)}\rho. $$
Let $r=r_{(0)}+p r_{(1)} $ and $j=j_{(0)}+p j_{(1)}$ be the $p$-adic expansions
of $r$ and $j$. Consider two cases.\\
Case 1: Assume $r_{(1)}+j_{(1)}\geq p$. 
It follows from (\ref{gets big}) and Corollary 
\ref{scaffold applied to valuation criterion} that
\begin{align}
v_2\left(\Psi^{(j)}\Psi^{(r)}\rho\right) & \geq 
p^2\ek+b_2+(j_{(1)}+r_{(1)}-p)b_2+(j_{(0)}+r_{(0)})pb_1 \nonumber \\
& = p^2\ek-(p+1)b_2+\mathfrak{b}(j)+\mathfrak{b}(r). \nonumber
\end{align}
Using $d_0\geq \frac{b_2-b_1}{p^2} $, $d_r\leq 
\frac{\mathfrak{b}(r) }{p^2} $ and 
$d_j\leq \frac{\mathfrak{b}(j)}{p^2} $ we find that
$$v_2\left(\pk^{d_0-d_j}
\Psi^{(j)}\rho_r\right)\geq p^2\ek-pb_2-b_1. $$
It is clear that if (\ref{second new bound}) is satisfied then
$p^2\ek-pb_2-b_1>0. $\\
Case 2: Assume $r_{(0)}+j_{(0)}\geq p$ and $r_{(1)}+j_{(1)}+1=p$.
Since $v_2(\Psi_1^p\rho)=2b_2 $ we see that
$$v_2\left(\Psi^{(r)}\Psi^{(j)}\rho
\right)\geq 2b_2+(r_{(1)}+j_{(1)})b_2+(j_{(0)}+r_{(0)}-p)pb_1=
\mathfrak{b}(r)+\mathfrak{b}(j)-p^2b_1. $$
Thus $v_2\left(\pk^{d_0-d_j}
\Psi^{(j)}\rho_r \right)\geq p^2d_0-p^2b_1>0. $
\end{proof}

\begin{proposition}\label{elements of assoc order}
Assume (\ref{second new bound}) is satisfied.
It is the case that
$$\pk^{-w_j}
\Psi^{(j)}\rho_r\in\mints $$
whenever $0\leq r,j<p^2.$
\end{proposition}
\begin{proof}
First note that Proposition \ref{big enough valuation} implies that
$\pk^{-w_j}\Psi^{(j)}\rho_r\in\mints $ whenever
$j+r\geq p^2.$ So we may assume that $j+r<p^2.$ First suppose that there is not a carry when adding the $p$-adic expansions of $j$ and $r$. Then
$\pk^{-w_j}\Psi^{(j)}\rho_r=\pk^{-w_j-d_r}\Psi^{(j)}\Psi^{(r)}\rho.
$ Note that $w_j\leq d_{j+r}-d_r$ so 
$$v_2\left(\pk^{-w_j-d_r}\Psi^{(j)}\Psi^{(r)}\rho \right)
\geq v_2\left(\pk^{-d_{j+r}}\Psi^{(j)}\Psi^{(r)}\rho \right).$$
Since there is not a carry when adding the $p$-adic expansions
of $j$ and $r$ we get $\pk^{-d_{j+r}}\Psi^{(j)}\Psi^{(r)}\rho=
\rho_{j+r}$. Thus 
$$v_2\left(\pk^{-w_j}\Psi^{(j)}\rho_r \right)\geq r(j+r)\geq 0. $$

Now let $j=j_{(0)}+p j_{(1)}$ and $r=r_{(0)}+p r_{(1)}$. Assume
that $j_{(0)}+r_{(0)}\geq p $ and $j_{(1)}+r_{(1)}+1<p. $
Now $$\Psi^{(j)}\Psi^{(r)}\rho=
\Psi_1^{j_{(0)}+r_{(0)}-p}\Psi_2^{j_{(1)}+r_{(1)}}\Psi_1^p\rho. $$
Since $v_2(\Psi_1^p\rho)\geq 2b_2$ we see that
$$v_2\left(\Psi^{(j)}\Psi^{(r)}\rho \right)\geq 
2b_2+(j_{(1)}+r_{(1)})b_2+(j_{(0)}+r_{(0)}-p)pb_1=\mathfrak{b}(j+r). $$
Finally, since $d_{j+r}=\lfloor \frac{\mathfrak{b}(j+r)}{p^2}\rfloor $,
it follows that $$v_2\left(\pk^{-w_j}\Psi^{(j)}\rho_r\right)\geq
v_2\left(\pk^{-d_{r+j}}\Psi^{(j)}\Psi^{(r)}\rho \right)\geq 0. $$
\end{proof}
\begin{proposition}\label{second congruence}
Assume (\ref{second new bound}) is satisfied.
Let $0\leq j,r<p^2 $. Then
\begin{equation}\label{congruence 2} 
\pk^{-w_j}\Psi^{(j)}\rho_r\equiv \pk^{d_{j+r}-d_r-w_j}\rho_{j+r} 
\mod\mmax^{p^2\ek-pb_2-(p^2-p+1)b_1}
\end{equation}
with equality if there is not a carry when adding the $p$-adic expansions
of $j$ and $r$.
\end{proposition}
\begin{proof}
Let  $j=j_{(0)}+p j_{(1)}$ and $r=r_{(0)}+p r_{(1)} $ be the $p$-adic expansions
of $j$ and $r$ respectively. Similar to what we saw in
Proposition \ref{new congruence} we see that
that equality holds
in (\ref{congruence 2})
when there is not a carry when adding the $p$-adic expansions of
$j$ and $r$. Moreover, if $j+r<p^2 $ then (\ref{congruence 2})
is implied by Proposition \ref{new congruence} since
$-w_j\geq d_0-d_j. $ So we may assume $j+r\geq p^2$. This leaves
us with two cases.

Case 1: Assume $j_{(1)}+r_{(1)}\geq p $.
Now $\rho_{j+r}=0 $. Consider 
$$\pk^{-w_j}\Psi^{(j)}\rho_r=\pk^{-d_r-w_j}\Psi^{(j)}\Psi^{(r)}\rho=
\pk^{-d_r-w_j}\Psi_1^{j_{(0)}+r_{(0)}}\Psi_2^{j_{(1)}+r_{(1)}-p}\Psi_2^p\rho. $$
It follows from (\ref{sigma 2}) and Corollary 
\ref{scaffold applied to valuation criterion} that
$$\Psi_1^{j_{(0)}+r_{(0)}}\Psi_2^{j_{(1)}+r_{(1)}-p}\Psi_2^p\rho\equiv 
0\mod \mmax^{p^2\ek-(p-1)b_2+(j_{(1)}+r_{(1)})b_2+(j_{(0)}+r_{(0)})pb_1}. $$
Now $p^2\ek-(p-1)b_2(j_{(1)}+r_{(1)})b_2+(j_{(0)}+r_{(0)})pb_1=
p^2\ek-(p+1)b_2+\mathfrak{b}(j)+\mathfrak{b}(r)$.
Using the bounds $-d_r\geq -\frac{\mathfrak{b}(r)}{p^2}  $
we see that
\begin{alignat}{3}
\pk^{-d_r-w_j}\Psi^{(j)}\Psi^{(r)}\rho & \equiv 0 &&
\mod \mmax^{p^2\ek-(p+1)b_2+p^2m} \nonumber \\
& \equiv 
0 && \mod \mmax^{p^2\ek-pb_2-b_1}. \nonumber 
\end{alignat}

Case 2: Assume $j_{(0)}+r_{(0)}\geq p $ and $j_{(1)}+r_{(1)}=p-1 $.
Again, $\rho_{j+r}=0 $. Consider
$$
\Psi^{(j)}\Psi^{(r)}\rho =
\Psi_1^{j_{(0)}+r_{(0)}-p}\Psi_2^{j_{(1)}+r_{(1)}}\Psi_1^p\rho
\equiv 
0 \mod \mmax^{p^2\ek-(p-1)b_2-p(p-1)b_1+(j_{(1)}+r_{(1)})b_2+(j_{(0)}+r_{(0)})pb_1}.
$$
Now $p^2\ek-(p-1)b_2-p(p-1)b_1+(j_{(1)}+r_{(1)})b_2+(j_{(0)}+r_{(0)})pb_1=
p^2\ek-(p+1)b_2-p(p-1)b_1+\mathfrak{b}(j)+\mathfrak{b}(r)$ and again
$-d_r\geq -\frac{\mathfrak{b}(r)}{p^2} $.
Hence
\begin{alignat}{3}
\pk^{-d_r-w_j}\Psi^{(j)}\Psi^{(r)}\rho & \equiv 0 &&
\mod \mmax^{p^2\ek-(p+1)b_2-p(p-1)b_1+p^2m} \nonumber \\
& \equiv 
0 && \mod \mmax^{p^2\ek-pb_2-(p^2-p+1)b_1}. \nonumber 
\end{alignat}
Clearly $p^2e_K-pb_2-(p^2-p+1)b_1<p^2\ek-pb_2-b_1 $ and so the claim holds.
\end{proof}
Notice that (\ref{second new bound}) can be restated as
$\ek>u_2+p^{-2}(b_2-(p^2-1)b_1)=u_2+p^{-2}b_1+p^{-2}(b_2-p^2b_1) $.
Hence (\ref{second new bound}) implies that 
$\ek\geq u_2+p^{-2}b_1+1$ since
$b_2>p^2b_1 $. Moreover, since
$p\nmid b_1$ we see that (\ref{second new bound}) implies that
$\ek>u_2+p^{-2}b_1+1$. Thus
(\ref{second new bound}) implies that
\begin{equation}\label{its huge}
    p^2\ek-pb_2-(p^2-p+1)>p^2.
\end{equation}
Now we arrive at our main result:
\begin{theorem}\label{main result} Let $\bigfield/\littlefield$ be the 
$C_{p^2}$-extension constructed in Section 3 using 
Choices \ref{first choice} and \ref{second choice}. Assume
further that the lower ramification numbers $b_1$ and $b_2$
satisfy
(\ref{second new bound}).\\
(a) The associated order $\asscorder$ of $\mints$
has $\kints $ basis $\{\pk^{-w_j}\Psi^{(j)}\}_{j=0}^{p^2-1}$.\\
(b) If $w_j=d_j-d_0$ for all $0\leq j\leq p^2-1$, then
$\mints$ is free over $\asscorder$; moreover,
$\mints=
\asscorder\cdot\rz$.\\
(c) Conversely, if $\mints$ is free over $\asscorder$ then
$w_j=d_j-d_0$ for all $0\leq j\leq p^2-1.$
\end{theorem}
\begin{proof}
(a) First notice that Proposition \ref{elements of assoc order} implies that
$\pk^{-w_j}\Psi^{(j)}\in\asscorder $ since $\{\rho_a\}_{a=0}^{p^2-1} $
is an $\kints$-basis for $\mints$.
Now given an element 
$\alpha\in \littlefield[G] $ we may write
$\alpha=\displaystyle\sum_{j=0}^{p^2-1}c_j\Psi^{(j)} $ with $c_j\in \littlefield$.
Since $\{\rho_a\}_{a=0}^{p^2-1} $ is a $\kints$-basis for $\mints$
we find that $\alpha\in \asscorder$ is equivalent to $\alpha\rho_a=
\displaystyle\sum_{j=0}^{p^2-1}c_j\Psi^{(j)}\rho_a\in 
\mints $ for all $0\leq a<p^2$.

Consider the case $a=0$. We have 
$$\sum_{j=0}^{p^2-1}c_j\Psi^{(j)}\rho_0=\sum_{j=0}^{p^2-1}c_j\pk^{d_j-d_0}\rho_j.$$
Since this is an element of $\mints$ it follows that
$v_0\left(c_j\pk^{d_j-d_0}\right)\geq 0$ for each $0\leq j<p^2.$

It suffices to show that $w_j\geq -v_0(c_j)$ whenever 
$0\leq j\leq p^2-1$.
To this end observe that
if $a+j\geq p^2$ then Proposition \ref{big enough valuation} implies that 
$c_j\Psi^{(j)}\rho_a\in \mints $ since 
$v_0(c_j)\geq d_0-d_j.$ for all $0\leq j\leq p^2-1.$
Now it follows from Proposition \ref{new congruence} and (\ref{its huge}) that
$$\alpha\rho_a \equiv 
\sum_{0\leq j< p^2-a}c_j\Psi^{(j)}\rho_a
 \equiv \sum_{0\leq j<p^2-a}c_j\pk^{d_{j+a}-d_a}\rho_{j+a} 
 \mod \mints.$$
Since  
$\alpha\rho_a\in\mints$
it follows that $c_j\pk^{d_{j+a}-d_a}\in\kints $ whenever $a+j<p^2$.
This means that $d_{j+a}-d_a\geq -v_0(c_j)$ whenever $0\leq a<p^2-j$
and so $w_j\geq -v_0(c_j) $.\\
(b)  Suppose that $w_j=d_j-d_0 $ for all $j$. Consider
$$\pk^{-w_j}\Psi^{(j)}\rz=
\pk^{d_0-d_j}\Psi^{(j)}\pk^{-d_0}\Psi^{(0)}\rho=
\pk^{-d_j}\Psi^{(j)}\rho=\rho_j. $$
Hence the $\kints$-basis 
$\{\pk^{-w_j}\Psi^{(j)} \}_{j=0}^{p^2-1} $ of $\asscorder$
takes $\rz$ to the basis $\{\rho_j\}_{j=0}^{p^2-1} $ of
$\mints$. Thus $\mints$
is a free $\asscorder$-module.\\
(c) Assume that $\mints=\asscorder\cdot\nu$ for some
$\nu\in\mints$. Since $\{\rho_r \}_{r=0}^{p^2-1}$ is an
$\kints$-basis for $\mints$, we may write
$\nu=\displaystyle\sum_{r=0}^{p^2-1}x_r\rho_r $
with $x_r\in\kints$.
Now $\{\pk^{-w_j}\Psi^{(j)}\nu \}_{j=0}^{p^2-1} $ is also
an $\kints$-basis for $\mints$.
It follows from Proposition \ref{second congruence} that 
$$\pk^{-w_j}\Psi^{(j)}\nu=\sum_{r=0}^{p^2-1}
x_r\pk^{-w_j}\Psi^{(j)}\rho_r\equiv
\sum_{r=0}^{p^2-1}x_r\pk^{d_{j+r}-d_{r}-w_j}\rho_{j+r}\mod\mmax^{p^2\ek-pb_2-(p^2-p+1)b_1}.$$
Hence (\ref{its huge}) implies that 
$$\pk^{-w_j}\Psi^{(j)}\nu\equiv
\sum_{r=0}^{p^2-1}x_r\pk^{d_{j+r}-d_{r}-w_j}\rho_{j+r}
\mod \pk\mmax.$$
For $0\leq i,j\leq p^2-1$ let $$b_{ij}=\left\{\begin{array}{cc} 
x_{j-i}\pk^{d_{j}-d_{j-i}-w_{i}}, & i\leq j\\
0, & i>j. \end{array}\right. $$
Let $B:=(b_{ij})_{0\leq i,j\leq p^2-1}$. Let $A\in M_{p^2}(\kints)$ be the change of coordinates matrix
taking the $\kints$-basis 
$\{\rho_j\}_{j=0}^{p^2-1} $ to the $\kints$-basis
$\{\pk^{-w_j}\Psi^{(j)}\nu\}_{j=0}^{p^2-1} $, 
say $A=(a_{ij})_{0\leq i,j\leq p^2-1}$.
We see that $a_{ij}\equiv b_{ij} \mod \kmax $ and so
$\det(A)\equiv \det(B)\mod\kmax$. 
Since $B$ is upper-triangular we see that $$\det(B)=
x_0^{p^2}\prod_{j=0}^{p^2-1}\pk^{d_j-d_0-w_j}.$$
We also know that 
$v_0(\det(A))=0 $. Thus
$$0=v_0(\det(B))=p^2v_0(x_0)+\sum_{j=0}^{p^2-1}d_j-d_0-w_j. $$
Since each term in the sum is non-negative we conclude that
$w_j=d_j-d_0 $ for each $0\leq j\leq p^2-1$.
\end{proof}

As a Corollary to this we get:
\begin{corollary} Let $\bigfield/\littlefield$ satisfy the conditions
of Theorem \ref{main result}.
Let $r(b_2)$ be the least non-negative residue of $b_2$ modulo $p^2$.
Let $\asscorder=\{\alpha\in \littlefield[G]:\alpha\mints\subseteq 
\mints\}$ be the associated order of $\mints$.
Then $\mints $ is free over $\asscorder $ if and only if
$r(b_2)\mid p^2-1 $. Furthermore, if $\mints$ is free over
$\asscorder$ then $\mints=\asscorder\cdot\rz$ for any
$\rz\in\bigfield$ such that $v_2(\rz)=r(b_2)$.
\end{corollary}
\begin{proof}
Follows the proof of Theorem 1.1 in [BE] given on pages 3604-3606.
\end{proof}
\section{An Example}
Let $\littlefield=\mathbb{Q}_3(\sqrt[6]{3})$. Let $\bigfield=\littlefield(x_1,x_2)$ where
$(x_1^3,x_2^3)=(x_1,x_2)\oplus (\pk^{-1},\pk^{-4}). $
That is, $x_1^3-x_1=\pk^{-1} $ and $x_2^3-x_2=\pk^{-4}-\pk^{-2}x_1-\pk^{-1}x_1^2 $.
Here $p=3$, $\ek=6$, $a_1=\pk^{-1}$, $a_2=\pk^{-4}$ and $b_1=1=m$.
We verify the following:
\begin{enumerate}
\item $-\frac{p}{p^2-1}\ek=-\frac{18}{8}<-1=v_0(a_1)$
\item $v_0(a_1)+\frac{p-1}{p}b_1=-\frac{10}{3}<-3=v_0(a_1^p) $
\item $p^2\ek=54>38=(p+1)b_2-(p-1)b_1 $
\item $\ek=6>\frac{46}{9}=u_2+p^{-2}b_1+1 $.
\end{enumerate}
So $\bigfield/K $ is a totally ramified, 
cyclic extension of degree $p^2$ which 
has a Galois scaffold. Here
$r(b)=1\mid p^2-1$ so Theorem 4.3 tells us that
$\mints=\asscorder\cdot\pim $.
We will verify this.

For $0\leq a\leq 8 $ we see that
$$\mathfrak{b}(a)=\left\{\begin{array}{ccc}
1, & \mbox{if} & 0\leq a\leq 2\\
2, & \mbox{if} & 3\leq a\leq 5\\
3, & \mbox{if} & 6\leq a\leq 7\\
4, & \mbox{if} & a=8
\end{array}\right. $$
and one can verify that
$$w_j=\left\{\begin{array}{ccc}
0, & \mbox{if} & 0\leq j\leq 2\\
1, & \mbox{if} & 3\leq j\leq 5\\
2, & \mbox{if} & 6\leq j\leq 7\\
3, & \mbox{if} & j=8.
\end{array}\right. $$
Notice that $w_j=d_j-d_0$ for $0\leq j\leq 8$. We know that
$\mathcal{S}:=\{\pk^{-w_j}\Psi^{(j)}\}_{0\leq j\leq 8}$
is an $\kints$-basis for
$\asscorder$.
We see that $\mathcal{S}=\{1,\Psi_1,\Psi_1^2, \pk^{-1}\Psi_2,
\pk^{-1}\Psi_1\Psi_2,\pk^{-1}\Psi_1^2\Psi_2,
\pk^{-2}\Psi_2^2,\pk^{-2}\Psi_1\Psi_2^2,
\pk^{-3}\Psi_1^2\Psi_2^2\}$. Now we consider 
$\left\{v_2\left(\pk^{-w_j}\Psi^{(j)}\pim\right)\right\}_{0\leq j\leq 8} $. If it
is a full set of residues modulo $9$ then $\pim$ is a free
generator for $\mints$ over $\asscorder$.
Indeed:
$$\begin{array}{ccc}v_2(\pk^{-3}\Psi_1^2\Psi_2^2\pim)=0 &
v_2(1\pim)=1 & v_2(\pk^{-1}\Psi_2\pim)=2\\
v_2(\pk^{-2}\Psi_2^2\pim)=3 &
v_2(\Psi_1\pim)=4 & v_2(\pk^{-1}\Psi_1\Psi_2\pim)=5\\
v_2(\pk^{-2}\Psi_1\Psi_2^2\pim)=6 &
v_2(\Psi_1^2\pim)=7 & v_2(\pk^{-1}\Psi_1^2\Psi_2\pim)=8
\end{array}$$ 

\newpage 
\noindent\textbf{References}\\
\noindent [BCE] Nigel P. Byott, Lindsay N. Childs,
G. Griffith Elder, Scaffolds and Generalized Integral Module Structure, Tome 68, no 3 (2018) 965-1010.\\
\noindent [BE13] Nigel P. Byott, G. Griffith. Elder, Galois scaffolds and Galois module structure in extensions of characteristic $p$ local fields of degree $p^2$, Journal of Number Theory 133 (2013) 3598-3610.\\
\noindent [BE18] N. P. Byott, G. G. Elder, 
Sufficient Conditions For Large Galois Scaffolds,
Journal of Number Theory 182 (2018) 95-130.\\
\noindent [E] G. Griffith Elder, A valuation criterion for normal basis generators in local fields of characteristic $p$, Arch. Math. 94 (2010), 43-47.\\
\noindent [FV] I.B. Fesenko, S. V. Vostokov, Local Fields and 
Their Extensions (2nd edition) Translations of mathematical monographs. 1993. American Mathematical Society.\\
\noindent [H] O. Hyodo, Wild Ramification in the Imperfect Residue Field Case,
Advanced Studies in Mathematics 12, 1987, Galois Representations and Arithmetic 
Algebraic Geometry pp. 287-314.\\
\noindent [MW] R. E. Mackenzie, G. Whaples,
Artin-Schreier Equations in Characteristic Zero, 
American Journal of Mathematics, Vol 78, No. 3 (July 1956), pp. 473-485.
The Johns Hopkins University Press.\\
\noindent [VZ] S. V. Vostokov, I. B. Zhukov, Some approaches to the construction of abelian extensions for $p$-adic fields, Amer. Math Soc. Transl. (2) Vol. 166, 1995.\\

\end{document}